\documentclass{article}

\usepackage{arxiv}

\usepackage[utf8]{inputenc} 
\usepackage[T1]{fontenc}    
\usepackage{hyperref}       
\usepackage{url}            
\usepackage{booktabs}       
\usepackage{amsfonts}       
\usepackage{nicefrac}       
\usepackage{microtype}      
\usepackage{lipsum}
\usepackage{amsmath}

\usepackage{amsfonts,amssymb}
\usepackage{graphicx}
\usepackage{amsthm}
\usepackage{xcolor}

\newcommand{\R}{\mathbb{R}}

\newtheorem{theorem}{Theorem}[section]
\newtheorem{proposition}{Proposition}[section]
\newtheorem{lemma}{Lemma}[section]
\newtheorem{corollary}{Corollary}[section]

\newtheorem{remark}{Remark}[section]

\usepackage{cite}
\usepackage{authblk}

\def\sign{\text{sgn}\,}
\def\supp{\text{supp}\,}
\newcommand{\p}{\partial}
\newcommand{\bb}{\begin{equation}}
\newcommand{\ee}{\end{equation}}
\newcommand{\ba}{\begin{array}}
\newcommand{\ea}{\end{array}}
\newcommand{\f}{\frac}
\usepackage[all]{xy}
\newcommand{\ds}{\displaystyle}

\newcommand{\al}{\alpha}

\numberwithin{equation}{section}

\title{Conserved quantities, continuation and compactly supported solutions of some shallow water models} 

\author{Igor Leite Freire}

\affil{Centro de Matem\'atica, Computa\c{c}\~ao e Cogni\c c\~ao,\\ Universidade Federal do ABC,\\ Avenida dos Estados, $5001$, Bairro Bangu,\\
$09.210-580$, Santo Andr\'e, SP - Brasil\\
  \texttt{igor.freire@ufabc.edu.br} \\
  \texttt{igor.leite.freire@gmail.com}}

\begin{document}
\maketitle
\begin{abstract}
\centering\begin{minipage}{\dimexpr\paperwidth-9cm}
\textbf{Abstract:} 
A proof that strong solutions of the Dullin-Gottwald-Holm equation vanishing on an open set of the (1+1) space-time are identically zero is presented. In order to do it, we use a geometrical approach based on the conserved quantities of the equation to prove a unique continuation result for its solutions. We show that this idea can be applied to a large class of equations of the Camassa-Holm type.
\vspace{0.1cm}

\textbf{2010 AMS Mathematics Classification numbers}: 35Q35, 35Q51, 37K10.


\end{minipage}
\end{abstract}
\bigskip
\newpage
\tableofcontents
\newpage

\section{Introduction}

Among the myriad of models describing shallow water waves, two equations are among the most prominent: The Korteweg-de Vries (KdV) \cite{kdv} and the Camassa-Holm (CH) \cite{chprl} equations. Also, they are prototype equations in the field of integrable systems and have an extensive list of nice mathematical properties, such as: Their solutions are well-posed in the Hadamard sense \cite{const1998-1, const1998-2, const1998-3,const2000-1,const2000-2, blanco,kato,katoma}; their solutions describe surfaces with constant curvature \cite{chern,reyes2002}; and they also have algebraic properties \cite{strachan,novikov}. See \cite{raspajde} for a wide discussion about the properties of the CH and KdV equations.

Although the KdV equation was derived more than a century ago \cite{kdv}, from the point of view of mathematical modelling, it fails miserably at describing wave breaking in shallow water \cite[page 476]{whi}, which is one of the most common and interesting phenomena in the studies of water waves. On the other hand, the CH equation not only seems to be a more realistic model to describe waves in water models, at least from the point of view of blow-up of its solutions, see \cite{const1998-1,const1998-2,const1998-3,const2000-1,escher}, but also it admits richer types of travelling waves than the KdV equation, see \cite{raspa} and references therein. 

It is nowadays well known that if a non-trivial initial data ({\it i.e}, a non-identically vanishing function) of the CH equation is compactly supported, then the corresponding solution loses this property instantly. This was firstly shown to strong solutions of the CH equation by Constantin \cite{constjmp} and subsequently by Henry \cite{henry}. Furthermore, it emerges from their results that if a strong solution of the Camassa-Holm equation is compactly supported at some $t>0$, then it necessarily must be identically null.

Later, in \cite{himcmp} it was shown that any non-trivial strong solution of the CH equation, corresponding to an initial data with fast decay at infinity, behaves asymptotically as a peakon solution for each fixed $t>0$. Moreover, non-trivial compactly supported data with enough regularity gives rise to strong solutions blowing-up at finite time, see \cite[Corollary 1.1]{himcmp}. More recently, Linares and Ponce \cite{linares} proved a unique continuation result for the CH equation, which says that if a solution of the CH equation vanishes on a non-empty open set, then it is null everywhere.

Let us introduce the object of investigation in the present work: the Dullin-Gottwald-Holm (DGH) equation 
\bb\label{1.0.1}
u_t-\al^2u_{txx}+2\omega u_x+3uu_x+\gamma u_{xxx}=\al^2(2u_xu_{xx}+uu_{xxx}),
\ee
which comes from the physics of fluids and is named after the work by Dullin, Gottwald and Holm \cite{dgh}. Above, $\al,\, \omega$ and $\gamma$ are parameters related to the the physical derivation of the model, see \cite{chprl,dgh}. We assume that $\al$ is non-negative, although the paper is mostly concerned with $\al>0$. Note that the KdV and CH equations are obtained from \eqref{1.0.1} by taking $\al=0$ and $\gamma=0$, respectively.

We observe that we can rewrite the DGH equation as
$$
m_t+2\omega u_x+um_x+2u_xm+\gamma u_{xxx}=0,
$$
where $m:=u-\al^2u_{xx}$ is known as momentum \cite{chprl,dgh}. In \eqref{1.0.1}, and consequently, in the equation above, the variables $t$ and $x$ are regarded as time and space in view of the physical meaning behind the equation and we shall maintain this terminology and meaning in the manuscript. 

One advantage of the last representation of the DGH equation is that we can realise one of its most utterly important facts: it is bi-Hamiltonian, that is, it can be written in two compatible Hamiltonian forms \cite{dgh} (see \cite[Chapter 7]{olverbook} for further details about bi-Hamiltonian systems)
$$
m_t=-B_2\f{\delta {\cal H}_1}{\delta m}=-B_1\f{\delta {\cal H}_2}{\delta m},
$$
where
$$
B_1=\p_x-\al^2\p_x^3,\quad B_2=\p_x(m\cdot)+m\p_x(\cdot)+2\omega\p_x(\cdot)+\gamma \p_x^3(\cdot),
$$
$\p_x(m\cdot)f:=\p_x(mf)$, etc, and
$$
{\cal H}_1=\int_{\R}(u^2+\al^2u_x^2)dx,\quad {\cal H}_2=\int_\R(u^3+\al^2u_x^3+2\omega u^2-\gamma u_x^2)dx.
$$

The quantities ${\cal H}_1$ and ${\cal H}_2$ above are conserved quantities for the DGH equation, meaning that they are invariant with respect to time and, in particular, as long as $\al\neq0$, the quantity ${\cal H}_1$ mathematically corresponds to the square of the Sobolev norm\footnote{Formally speaking, a weighted Sobolev norm for $\al>0$ and $\al\neq1$.} $\|\cdot\|_{H^1(\R)}$ of the solutions of the equation, and some of its derivatives,  decaying to $0$ as $|x|\rightarrow\infty$.

The invariance of the Sobolev norm is a key ingredient for proving several qualitative results of the solutions of both CH and DGH equations. Moreover, the latter equation shares with the former the fact that they have local well-posed solutions in the Hadamard's sense in Sobolev spaces, at least whenever certain relations between the coefficients $\omega$ and $\gamma$ are satisfied, see \cite{liu,tiancmp2005,zhoujfa2007,zhou} and  \cite{mustafa,zhoujfa2007,zhou} for the real and periodic cases, respectively.

Still about similarities, the solutions of the DGH equation also provide metrics for surfaces with constant curvature \cite[Theorem 4]{raspaspam}. On the other hand, as far as the author knows, none of the results proved in \cite{constjmp,henry,linares} have been investigated for the DGH equation.

The first motivation of this work was the investigation of whether non-trivial (strong) solutions of the DGH equation might be compactly supported. In this sense our goal was originally the extension of the results due to Constantin \cite{constjmp} and Henry \cite{henry} to the DGH equation.

In the course of the investigation mentioned above, the author discovered the work by Linares and Ponce \cite{linares}. This brought to the original investigation the following question: Would it be possible to extend to the DGH equation a similar result? Of course the results proved in \cite{linares} gave a strong indication that the answer would be positive. However, what is reported in this work goes far beyond a mere application or adaption of a established result to a different equation. What is really presented here is a different framework for proving unique continuation for solutions of non-local evolution equations satisfying certain conditions.

The main difference between what it is done in the present work and that made in \cite{linares} is the use of conserved quantities: Both CH and DGH equations are models coming from Physics and, therefore, some of their conserved quantities are physically relevant. For example, the conservation of the Sobolev norm corresponds to the conservation of energy of the model. Also, these equations themselves are conservation laws, which implies that the integral of a non-negative solution $u$ over the whole domain corresponds to the mass conservation. These observations are used to develop a different framework for investigating unique continuation properties of solutions of the DGH equation and related models.

Our main argument is quite simple and comes from Physics: If the mass or energy of a conservative physical system vanishes at some time $t$, then it must necessarily vanish at all and, therefore, we do not have the occurrence of the phenomenon. Mathematically speaking, this means that the solutions of the equations describing the physical process, in our case, the DGH equation, must necessarily vanish.

In the next section it is presented the main results of this paper and its outline.

\section{Main results of the paper and its outline}\label{sec2}

In this section the main results of the paper and its outline are presented, as well as some comments showing how they are inserted in the state of the art in the field. We also discuss the novelty of the work and its innovations.

As said before, the main focus of this work is the equation \eqref{1.0.1} with $\al\neq0$. Under a suitable scaling of the variables, and eventually renaming the constants, we can choose $\al=1$ in \eqref{1.0.1} and this is henceforth assumed throughout this paper, but for Section \ref{sec6}. Some of our results are concerned with the Cauchy problem on the real line
\bb\label{1.0.2}
\left\{
\ba{ll}
m_t+2\omega u_x+um_x+2u_xm=-\gamma u_{xxx},&t>0,\,\,x\in\R,\\
\\
u(0,x)=u_0(x),& x\in\R,
\ea
\right.
\ee
where $m:=u-u_{xx}$.

The initial value problem \eqref{1.0.2} is locally  well-posed provided that the initial data $u_0\in H^s(\R)$, with $s\geq 3/2$, see \cite[Proposition 2.1]{tiancmp2005}. If we add the condition $m_0+\omega+\gamma/2\geq0$, where $m_0:=u_0-u_0''$, not only we have $\|u_x\|_{L^\infty}<K$, for a certain positive constant $K$, but we also extend the local well-posedness to global. In \cite[Theorem 2.2]{liu}, the conditions on $m_0$ are relaxed, but the price payed is the imposition of the condition $\omega+\gamma/2=0$. The blow-up of the solutions of \eqref{1.0.2} was also investigated in \cite[Theorem 2.3]{tiancmp2005} and \cite[Theorem 3.1]{liu}. See also \cite{yin} for further results about the DGH equation in the real line.

The comments above support the view that the constants $\omega$ and $\gamma$ in \eqref{1.0.2} are relevant in the qualitative analysis of solutions of the problem \eqref{1.0.2}.

As earlier mentioned, one of our purposes in this work is to enlighten whether the solutions of the DGH equation are compactly supported and our first result in this direction is:

\begin{theorem}\label{teo1.1}
Let $u_0\in H^4(\R)$ be an initial data of the problem \eqref{1.0.2}, with corresponding solution $u$. If $m_0(x)+\omega+\gamma/2$ has compact support, then $m(t,\cdot)+\omega+\gamma/2$ has the same property, for every $t>0$ belonging to the interval of existence of the solution. 
\end{theorem}

The analogous of Theorem \ref{teo1.1} in the CH case is obtained when $\omega=\gamma=0$, see \cite[Proposition 1]{henry} and \cite[Theorem in section II]{constjmp}. It is worth mentioning that the result of Theorem \ref{teo1.1} is in line with the preceding observations, in the sense that the quantity $\omega+\gamma/2$ plays great relevance in our analysis. 

Theorem \ref{teo1.1} is proved in Section \ref{sec3}, where we also revisit some relevant facts about the DGH equation to set the wheels in motion for proving some technical results that will be used throughout the work.

The second result of this work can now be formulated:
\begin{theorem}\label{teo1.2}
Let $u_0\in H^{4}(\R)$ be a given initial data of the Cauchy problem \eqref{1.0.2}, $u$ its corresponding solution and $T$ its lifespan. The following affirmations are equivalent:
\begin{enumerate}
    \item For each $t\in[0,T)$, the function $x\mapsto u(t,x)+\omega+\gamma/2$ has compact support;
    \item $u\equiv0$ and $\omega+\gamma/2=0$;
    \item There exists a non-empty open set $\Omega\subseteq[0,T)\times\R$ such that $u(t,x)=0$, for all $(t,x)\in\Omega$, and $\omega+\gamma/2=0$.
\end{enumerate}
\end{theorem}

We observe that if $u_0$ is non-trivial and compactly supported, then necessarily $u(t,\cdot)$ is not compactly supported, for each $t>0$. If this happens, Theorem \ref{teo1.2} combined with \cite[Corollary 1.1]{himcmp} prove the following immediate consequence:
\begin{corollary}\label{cor}
Assume that $\gamma=-2\omega$. If for some $T>0$ and $s\geq 4$ a function $u\in C([0,T];H^s(\R))$ is a solution of \eqref{1.0.2} such that $u_0$ is non-trivial and compactly supported, then there exists $T^\ast\in(0,T)$ such that $u\in C([0,T^\ast);H^s(\R))$ and
$$
\int_0^{T^\ast}\|u_x(t,x)\|_{L^\infty(\R)}dt=\infty.
$$
\end{corollary}

In order to prove Theorem \ref{teo1.2} we use the ideas in \cite{linares} to show that if a solution of the DGH vanishes on an open set, we can construct an auxiliary function depending only on space and use the convolution to show that this function vanishes identically. Once this is established, and differently from \cite{linares}, we use conserved quantities to show that any solution $u$ of the DGH equation vanishing on open sets is necessarily null. In Section \ref{sec7} this idea is better explored and it is shown that if it is possible to construct such a function, in fact it is possible to find a one-parameter family of real functions, whose parameter is the time, vanishing identically. As a consequence, if the solution vanishes on an open set, then it vanishes on a strip containing this open set.

Our next result is:
\begin{theorem}\label{teo1.3}
Assume that $u$ is a solution of the DGH equation such that $u$ is bounded, $(\p_x\Lambda^{-2}u)(t,\cdot)\in C^1(\R)$, for each $t$ such that the solution exists, and $\|u(t,\cdot)\|_{H^1(\R)}$ is invariant. If $\omega+\gamma/2=0$ and the function $x\mapsto u(t,x)$ has compact support, for each $t\in[0,T)$, then $u\equiv0$.
\end{theorem}
 
The demonstration of Theorem \ref{teo1.3} given in this work can be applied to the CH equation in the real line and, actually, it is rather different of those in \cite{constjmp,henry}.

The proof of both theorems \ref{teo1.2} and \ref{teo1.3} is presented in Section \ref{sec4}.

So far the results reported have been concerned with real problems related to the DHG equation, although it is well known that the study of the periodic DGH equation is also active, see \cite{mustafa, zhoujfa2007,zhou}. Hence, it is natural to ask whether the results above might have periodic counter-parts.

Before answering this question, we recall the following equivalence relation between real numbers: $x\sim y$ if and only if $x-y$ is an integer number. Then, the quotient space $\R/\mathbb{Z}\approx[0,1)$ can be identified with the circle $\mathbb{S}$ (sometimes also referred as one-dimensional torus). 

\begin{theorem}\label{teo1.4}
Let $u_0\in H^4(\mathbb{S})$ be a given initial data of the Cauchy problem
\bb\label{1.0.3}
\left\{
\ba{ll}
m_t+2\omega u_x+um_x+2u_xm=-\gamma u_{xxx},& t>0,\,\,x\in\R,\\
\\
u(0,x)=u_0(x), & x\in\R,\\
\\
u(t,x+1)=u(t,x), & t\geq 0,\,\,x\in\R.
\ea
\right.
\ee
Let $u=u(t,x)$ be the corresponding solution of \eqref{1.0.3} with lifespan $T>0$. If $\omega+\gamma/2=0$ and the function $x\mapsto u(t,x)$ is compactly supported, for each $t\in[0,T)$, then $u\equiv0$. In particular, $u_0\equiv0$.
\end{theorem}

Theorem \ref{teo1.4} is actually a consequence of the following stronger result, regarding unique continuation of the solutions of \eqref{1.0.3}:

\begin{theorem}\label{teo1.5}
Let $u$ be a bounded solution of the DGH equation, with lifespan $T>0$, satisfying the periodic condition $u(t,x+1)=u(t,x)$, such that $(\p_x\Lambda^{-2}u)(t,\cdot)\in C^1(\mathbb{S})$ and $\|u(t,\cdot)\|_{H^1(\mathbb{S})}$ is invariant, for each $t$ such that the solution exists. If there exists an open set $\Omega\neq\emptyset$ such that $\left.u\right|_{\Omega}\equiv0$, then $u$ vanishes identically. 
\end{theorem}

The proof of theorems \ref{teo1.4} and \ref{teo1.5} are given in Section \ref{sec5}.

The remaining of the paper is as follows: In Section \ref{sec6} we show how to extend our results to weakly dissipative models considered in \cite{novjmp2013,wei,wujde2009,zhangjmp2015}. In Section \ref{sec7} we explore the main idea introduced in sections \ref{sec4} and \ref{sec5} and show how we can obtain unique continuation results for the solutions of equations of the CH and DGH types using their conserved quantities. A discussion of our results is presented in Section \ref{sec8}, while in Section \ref{sec9} we present our conclusions.

{\bf Novelty of the work and its contributions.} The first point to be mentioned is a unique continuation result for both real and periodic initial value problems related to the DGH equation, given in theorems \ref{teo1.2} and \ref{teo1.5}, respectively. As a consequence, we show that the unique compactly supported solution of these problems is the identically zero or, which is equivalent, in case these problems have an initial data compactly supported, the solution loses the latter property instantly. Our approach to prove the uniqueness of compactly supported solutions can be directly applied, for instance, to the Camassa-Holm equation and it is different from those methods used in \cite{constjmp, henry}. Also, for the unique continuation results, not only what is done here is new considering the specific problem, but it is also different of that reported in \cite{linares} for the Camassa-Holm equation, since in our case we use one of the Hamiltonians of the equation (corresponding to the invariance of the Sobolev norm) to prove a unique continuation of its solutions. A third feature of the present work is the fact that we extend in an intrinsic way our results to some models related to both CH and DGH equations.

{\bf Notation and conventions.} Here we fix and introduce the notations and conventions used along the work. By $\mathbb{E}$ we denote either $\R$ or $\mathbb{S}$. The space of the integrable functions on $\mathbb{E}$ is denoted by $L^1(\mathbb{E})$, while the Hilbert Sobolev space of order $s$ is denoted by $H^s(\mathbb{E})$, for each $s\in\mathbb{R}$. If $f$ and $g$ are two integrable functions, their convolution is denoted by $f\ast g$. If $u=u(t,x)$, we denote by $u_0(x)$ the function $x\mapsto u(0,x)$, $m=u-u_{xx}$ and $m_0=u_0-u_0''$. Given a function $u$ depending on two variables, unless otherwise stated, $u_t$ and $u_x$ denotes the derivatives of $u$ with respect to the first and the second arguments, respectively. We recall the Helmholtz operator $\Lambda^2=1-\p_x^2$ and its inverse $\Lambda^{-2}=(1-\p_x^2)^{-1}$, given by $\Lambda^{-2}(f)=g\ast f$, where $g(x)=e^{-|x|}/2$ in the real case whereas 
$$g(x)=\f{\cosh{(x-\lfloor x \rfloor-1/2)}}{2\sinh(1/2)}$$
in the periodic case. Above, $\lfloor \cdot \rfloor$ denotes the greatest integer function. The set of continuous and bounded functions is denoted by $C_b^0$.

\section{Proof of Theorem \ref{teo1.1} and technical results}\label{sec3}

We begin with some known facts from the literature and then prove technical ones that will be of great importance to pursue our goals.

\begin{lemma}\label{lema2.1}\textsc{(Local Well-Posedness).}
Given $u_0\in H^s(\mathbb{E})$, $s>3/2$, there exists a maximal time $T>0$, and a unique solution $u$ to the problem \eqref{1.0.2} such that $u\in C^0\left([0,T);H^s(\mathbb{E})\right)\cap C^1\left([0,T);H^{s-1}(\mathbb{E})\right)$. Moreover, such solution is continuously dependent on the initial data.
\end{lemma}

The proof of Lemma \ref{lema2.1} in the real case can be found in \cite[Proposition 2.1]{tiancmp2005}. For the periodic case, see \cite[Theorem 2.1]{zhoujfa2007}, \cite[Theorem 2.1]{zhou} and references thereof. In \cite{mustafa} a similar result is proved, but assuming that $u_0\in C^1(\mathbb{S})$, which implies $u\in C^0\left([0,T);C^1(\mathbb{S}\right)\cap C^1\left([0,T);C^0(\mathbb{S}\right))$.

The next results will be helpful in the demonstrations of theorems \ref{teo1.1} and \ref{teo1.2} and, for this reason, henceforth in the remaining of this section we only consider the real case.

\begin{lemma}\label{lema2.2}\textsc{(\cite[Lemma 2.3]{liu}, \cite{tiancmp2005})}
Assume that $u_0\in H^3(\R)$ be an initial data of the problem \eqref{1.0.2}, $u$ its corresponding solution, $T>0$ the lifespan of $u$ and $m=u-u_{xx}$. Then
\bb\label{2.1}
m_0(x)+\omega+\f{\gamma}{2}=\left(m(t,q(t,x))+\omega+\f{\gamma}{2}\right)q_x(t,x)^2,\quad (t,x)\in[0,T)\times\R,
\ee
where $q(t,x)\in C^1\left([0,T)\times\R,\R\right)$ is the unique solution of the initial value problem
$$
\left\{
\ba{lcl}
\ds{\f{d q}{dt}}&=&\ds{u(t,q)-\gamma},\\
\\
q(0,x)&=&x.
\ea
\right.
$$
Moreover, for each $t\in[0,T)$, the function $q(t,\cdot)$ is an increasing diffeomorfism of the line, satisfying
$$q_x(t,x)=\exp{\left(\int_0^tu_x(s,q(t,s))ds\right)}>0.$$
\end{lemma}

We are in conditions to prove Theorem \ref{teo1.1} right now.

{\bf Proof of Theorem \ref{teo1.1}.} Let $u_0\in H^4(\R)$ be an initial data of the problem \eqref{1.0.2}, $u$ its corresponding solution and $m=u-u_{xx}$. Assume that $m_0+\omega+\gamma/2$ is supported in the interval $[a,b]$. In particular, $m_0(z)+\omega+\gamma/2=0$, provided that $z\leq a$ or $z\geq b$.

By Lemma \ref{lema2.2}, $q(t,\cdot)$ is an increasing diffeomorphism. Consequently, from \eqref{2.1} we conclude that $m(t,q(t,z))+\omega+\gamma/2$ vanishes, whether $z\leq q(t,a)$ or $z\geq q(t,b)$, which implies that $m(t,\cdot)+\omega+\gamma/2$ has its support in the interval $[q(t,a),q(t,b)]$. \qed

\begin{proposition}\label{prop2.1}
Let $u\in C^2(\R)\cap H^2(\R)$ be a solution of \eqref{1.0.2} such that $m+\omega+\gamma/2$ has compact support. Then $u+\omega+\gamma/2$ has compact support if and only if 
\bb\label{2.0.2}
\int_\R e^{\pm x}\left[m(t,x)+\omega+\f{\gamma}{2}\right]dx=0.
\ee
\end{proposition}

\begin{proof}
Assume that $u(t,\cdot)$ is a solution of \eqref{1.0.2} such that $m(t,\cdot)+\omega+\gamma/2$ is compactly supported. Then, there exists $L>0$ such that $m(t,x)+\omega+\gamma/2=0$ for $|x|>L$.

Let us assume that $u+\omega+\gamma/2$ is compactly supported. Without loss of generality, we can assume that $L$ is large enough to guarantee that $u(t,z)+\omega+\gamma/2=0$, provided that $|z|>L$.

Suppose firstly $x>L$. Then
$$
\ba{lcl}
\ds{2u(t,x)+2\omega+\gamma}&=&\ds{\int_\R e^{-|x-y|}\left[m(t,y)+\omega+\f{\gamma}{2}\right]dy=e^{-x}\int_{-\infty}^{-x} e^{y}\underbrace{\left[m(t,y)+\omega+\f{\gamma}{2}\right]}_{=0,\,\text{since}\,y<-L}dy}\\
\\
&+&\ds{e^{-x}\int_{-x}^x e^{y}\left[m(t,y)+\omega+\f{\gamma}{2}\right]dy+e^{x}\int_{x}^\infty e^{-y}\underbrace{\left[m(t,y)+\omega+\f{\gamma}{2}\right]}_{=0,\,\text{since}\,y>L}dy}.
\ea
$$

The first and the third integrals above are evaluated in the region $|y|>x$. Since $x>L$ and the support of $m(t,\cdot)+\omega+\gamma/2$ is in $[-L,L]$, these integrals vanish. Therefore, $u(t,x)+\omega+\gamma/2=0$ with $|x|>L$ if and only if
$$
e^{-x}\int_{-x}^xe^{y}\left[m(t,y)+\omega+\f{\gamma}{2}\right]dy=0\Longleftrightarrow \int_{-x}^xe^{y}\left[m(t,y)+\omega+\f{\gamma}{2}\right]dy=0,
$$
and, as a consequence,
$$
\int_\R e^{y}\left[m(t,y)+\omega+\f{\gamma}{2}\right]dy=\lim_{x\rightarrow\infty}\int_{-x}^xe^{y}\left[m(t,y)+\omega+\f{\gamma}{2}\right]dy=0.
$$

Let us now assume that $x<-L$. A similar calculation yields
$$
\ba{lcl}
2u(t,x)+2\omega+\gamma&=&\ds{\int_{\R}e^{-|x-y|}\left[m(t,y)+\omega+\f{\gamma}{2}\right]dy=e^{x}\int_{-\infty}^x e^{-y}\underbrace{\left[m(t,y)+\omega+\f{\gamma}{2}\right]}_{=0,\,\text{since}\,y<-L}dy}\\
\\
&+&\ds{e^{x}\int_{x}^{-x} e^{-y}\left[m(t,y)+\omega+\f{\gamma}{2}\right]dy+e^{x}\int_{-x}^\infty e^{-y}\underbrace{\left[m(t,y)+\omega+\f{\gamma}{2}\right]}_{=0,\,\text{since}\,y>L}dy}.
\ea
$$

An analogous argumentation leads to the conclusion
$$
\int_\R e^{-y}\left[m(t,y)+\omega+\f{\gamma}{2}\right]dy=\lim_{x\rightarrow\infty}\int^{-x}_xe^{-y}\left[m(t,y)+\omega+\f{\gamma}{2}\right]dy=0.
$$

Conversely, if $m(t,\cdot)+\omega+\gamma/2$ has support compact, for some $L>0$, we would have $m(t,x)+\omega+\gamma/2=0$ provided that $|x|>L$.

If we assume $x$ larger than $L$, we obtain
$$
\ba{lcl}
2u(t,x)+2\omega+\gamma&=&\ds{\int_\R e^{-|x-y|}\left[m(t,y)+\omega+\f{\gamma}{2}\right]dy=e^{-x}\int_{-\infty}^{-L} e^{y}\underbrace{\left[m(t,y)+\omega+\f{\gamma}{2}\right]}_{=0,\,\text{since}\,y<-L}dy}\\
\\
&+&\ds{e^{-x}\int_{-L}^L e^{y}\left[m(t,y)+\omega+\f{\gamma}{2}\right]dy+e^{x}\int^{\infty}_L e^{-y}\underbrace{\left[m(t,y)+\omega+\f{\gamma}{2}\right]}_{=0,\,\text{since}\,y>L}dy},
\ea
$$
which gives
\bb\label{2.0.3}
u(t,x)+\omega+\f{\gamma}{2}=\f{1}{2}e^{-x}\int_{-L}^L e^{y}\left[m(t,y)+\omega+\f{\gamma}{2}\right]dy.
\ee
A similar procedure reads, for $-x>L$,
\bb\label{2.0.4}
u(t,x)+\omega+\f{\gamma}{2}=\f{1}{2}e^{x}\int_{-L}^L e^{-y}\left[m(t,y)+\omega+\f{\gamma}{2}\right]dy.
\ee

If \eqref{2.0.2} holds and $m(t,x)+\omega+\gamma/2=0$ whenever $|x|>L$, then
$$
\int_{-L}^L e^{\pm x}\left[m(t,x)+\omega+\f{\gamma}{2}\right]dx=\int_{-\infty}^\infty e^{\pm x}\left[m(t,x)+\omega+\f{\gamma}{2}\right]dx=0.
$$
These observations, jointly with \eqref{2.0.3} and \eqref{2.0.4}, implies that $u(t,x)+\omega+\gamma/2$ when $|x|>L$, which concludes the proof.
\end{proof}

The next result is a trivial and immediate consequence of Proposition \ref{prop2.1}, which is spotlighted due to its crucial relevance in the proof of Theorem \ref{teo1.2}.
\begin{corollary}\label{cor2.1}
Assuming the conditions in Proposition \ref{prop2.1}, we have the identity
$$
\f{d}{dt}\int_\R e^{\pm x}\left[m(t,x)+\omega+\f{\gamma}{2}\right]dx=0.
$$
\end{corollary}

Let $f$ be an integrable function with compact support. The following relations will be of great importance in the next result:
\bb\label{2.0.5}
\int_\R e^{\pm x}f'(x)dx=\mp \int_\R e^{\pm x}f(x)dx.
\ee
\begin{proposition}\label{prop2.2}
If $u$ is a solution of \eqref{1.0.2} such that $x\mapsto u(t,x)+\omega+\gamma/2$ is compactly supported, then
$$
\f{d}{dt}\int_\R e^{\pm x}\left[m+\omega+\f{\gamma}{2}\right]dx=\int_\R e^{\pm x}\left[\left(u+\omega+\f{\gamma}{2}\right)^2+\f{u_x^2}{2}\right]dx.
$$
\end{proposition}
\begin{proof} 
Let us define $\tilde{u}(t,x):=u(t,x)+\omega+\gamma/2$, for each $(t,x)\in[0,T)\times\R$, and $\tilde{m}=m+\omega+\gamma/2$. Then $x\mapsto u(t,x)+\omega+\gamma/2$ has compact support if and only if $\tilde{u}$ is compactly supported. Therefore,
$$
m_t=-um_x-2u_xm-\gamma u_{xxx}-2\omega u_x
$$
if and only if
\bb\label{2.0.6}
\ba{lcl}
\tilde{m}_t&=&\ds{-\tilde{u}\tilde{m}_x-2\tilde{u}_x\tilde{m}+\left(\omega+\f{3}{2}\gamma\right)\tilde{u}_{x}-\left(\omega+\f{3}{2}\gamma\right)\tilde{u}_{xxx}}\\
\\
&=&\ds{-\p_x\left[\f{3}{2}\tilde{u}^2+\f{\tilde{u}_x}{2}-\left(\omega+\f{3}{2}\gamma\right)\tilde{u}\right]+\p_x^3\left[\f{\tilde{u}^2}{2}-\left(\omega+\f{3}{2}\gamma\tilde{u}\right)\right]}.
\ea
\ee
Taking \eqref{2.0.6} and \eqref{2.0.5} into account, we obtain
$$
\ba{l}
\ds{\f{d}{dt}\int_\R e^{x}\left[m+\omega+\f{\gamma}{2}\right]dx}=\ds{\f{d}{dt}\int_\R e^{x}\tilde{m}dx=\int_\R e^{x}\tilde{m}_tdx=\int_\R e^{\pm x}\p_x^3\left[\f{\tilde{u}^2}{2}-\left(\omega+\f{3}{2}\gamma\tilde{u}\right)\right]dx}\\
\\
-\ds{\int_\R e^{\pm x}\p_x\left[\f{3}{2}\tilde{u}^2+\f{\tilde{u}_x}{2}-\left(\omega+\f{3}{2}\gamma\right)\tilde{u}\right]dx}=\ds{\pm\int_\R e^{\pm x}\left(\tilde{u}^2+\f{\tilde{u}_x^2}{2}\right)dx,}
\ea
$$
and the result is obtained noting that $\tilde{u}=u+\omega+\gamma/2$ and $\tilde{u}_x=u_x$.
\end{proof}

\section{Proof of theorems \ref{teo1.2} and \ref{teo1.3}}\label{sec4}

We recall that the for solutions of the DGH equation decaying to $0$ at infinity, the functional 
\bb\label{3.0.1}
{\cal H}(t):=\f{1}{2}\int_\R\left[u(t,x)^2+u_{x}(t,x)^2\right]dx=\f{1}{2}\|u(t,\cdot)\|^2_{H^1(\R)}
\ee
is independent of time \cite{dgh}, and it reflects the conservation of energy of the solutions of the equation. Also, by the Sobolev Embedding Theorem, see \cite[p. 47]{linaresbook} and \cite[p. 317]{taylor}, if $u\in H^s(\R)$, with $s>1/2$, then $u$ is bounded, continuous, and the inequality $\|u\|_{L^\infty(\R)}\leq c\|u\|_{H^s(\R)}$ holds, for some $c>0$.

\begin{proposition}\label{prop3.1}
Assume that $u_0\in H^s(\R)$, $s>3/2$, is an initial data of the problem \eqref{1.0.2}, with corresponding solution $u$ and lifespan $T>0$. If there exists some $t_0\in[0,T)$ such that $u\big|_{\{t_0\}\times\R}\equiv0$, then $u\equiv0$.
\end{proposition}

\begin{proof}
Let us consider the functional ${\cal H}$ given by \eqref{3.0.1}. Its time independence means that ${\cal H}(t)={\cal H}(0)$, for all $t\in[0,T)$. If $x\mapsto u(t_0,\cdot)\equiv0$, then ${\cal H}(t_0)=0$, that is, ${\cal H}(t)=0$, for any $t\in[0,T)$. Therefore, $\|u\|_{L^\infty(\R)}=0$, as well as $u$.
\end{proof}

\begin{proposition}\label{prop3.2}
Let $a$ and $b$ real numbers such that $a<b$, and $S:\R\rightarrow\R$ given by $S(y)=\sign{(a-y)}e^{-|a-y|}-\sign{(b-y)}e^{-|b-y|}$. Then $S\in L^1(\R)$ and $S(y)>0$ for all $y\notin[a,b]$. 
\end{proposition}

\begin{proof}
It is clear that $S\in L^1(\R)$. What remains to be proved is that $S(y)>0$.

Since $y\notin[a,b]$, we have only two possible possibilities: $y<a$ or $y>b$. If $y<a$, then $-|a-y|>-|b-y|$, $e^{-|a-y|}>e^{-|b-y|}$ and both $\sign{(a-y)}$ and $\sign{(b-y)}$ are positive. Therefore, $-\sign{(b-y)}e^{-|b-y|}>-\sign{(a-y)}e^{-|a-y|}$.

In the second case, we have $e^{-|b-y|}>e^{-|a-y|}$ and both $-\sign{(a-y)}$ and $-\sign{(b-y)}$ are positive, which lead us to the same conclusion.
\end{proof}

Let us now prove Theorem \ref{teo1.2}. We would like to observe that $2\Rightarrow 1$ as well as $2\Rightarrow 3$ and for this reason we only need to prove the implications $1\Rightarrow 2$ and $3\Rightarrow 2$.

{\bf Proof that $1\Rightarrow 2$} Assume that $x\mapsto u(t,x)+\omega+\gamma/2$ has compact support. From Proposition \ref{prop2.2} and Corollary \ref{cor2.1} we have
$$
0=\f{d}{dt}\int_\R e^{\pm x}\left[m+\omega+\f{\gamma}{2}\right]dx=\int_\R e^{\pm x}\left[\left(u+\omega+\f{\gamma}{2}\right)^2+\f{u_x^2}{2}\right]dx,
$$
which implies $u+\omega+\gamma/2\equiv0$, that is, $u(t,x)=-(\omega+\gamma/2)$, everywhere the solution is defined. In particular, evaluating the solution at $t=0$ we are forced to conclude that $u_0(x)=-(\omega+\gamma/2)$. Since $u_0\in H^4(\R)$, then $u_0(x)\rightarrow0$ as $|x|\rightarrow\infty$, implying that $\omega+\gamma/2=0$ and, consequently, $u\equiv0$.

{\bf Proof that $3\Rightarrow 2$} It is enough to prove that $u\equiv0$.

Assume that $\Omega\subseteq[0,T)\times\R$ is an open set such that $(t,x)\in\Omega\Rightarrow\,u(t,x)=0$. Then, there exists real numbers $a,\,b$ and $t_0$ such that $\{t_0\}\times[a,b]\subseteq\Omega$, but $\{t_0\}\times\R\subsetneqq\Omega$. Otherwise,  the result is an immediate consequence of Proposition \ref{prop3.1}.

Let $f,\,F:\R\rightarrow\R$ given by 
$f(x)=u(t_0,x)^2+u_x^2(t_0,x)/2$ and $F(x)=(\p_x\Lambda^{-2}f)(x)$. Since $u(t_0,\cdot)$ is continuously differentiable, bounded, and both $u(t_0,\cdot)$ and $u_x(t_0,\cdot)$ are integrable, then $f\in C_b^0(\R)\cap L^1(\R)$ and $F\in C_b^0(\R)\cap C^1(\R)\cap L^1(\R)$. In particular,
$$
F(x)=\int_{\R}\p_x\left(\f{1}{2}e^{-|x-y|}f(y)\right)dy=\f{1}{2}\int_\R\left[-\sign{(x-y)}e^{-|x-y|}f(y)\right]dy.
$$

We firstly note that $f\equiv0$ if and only if $u(t_0,\cdot)\equiv0$. Therefore, if we show that $f$ vanishes identically, then Proposition \ref{prop3.1} implies that $u$ vanishes everywhere.

Let us proceed by contradiction, assuming that $f$ is not identically $0$. This means that $u(t_0,y)\neq0$ for some $y$, which implies the existence of an open set $V\subseteq\R$ such that $y\in V$ and $f(y)>0$ for all $y\in V$.

Let $S$ be the function given in Proposition \ref{prop3.2}. We note that
\bb\label{3.0.2}
\int_\R S(y)f(y)dy\geq\int_{V}S(y)f(y) >0.
\ee
Also, we have the inequalities
$$
F(b)=\f{1}{2}\int_\R\left[-\sign{(b-y)}e^{-|b-y|}f(y)\right]dy\geq\f{1}{2}\int_\R\left[-\sign{(a-y)}e^{-|a-y|}f(y)\right]dy=F(a),
$$
and, therefore,
\bb\label{3.0.3}
0\leq\f{1}{2}\int_\R\underbrace{\left[\sign{(a-y)}e^{-|a-y|}-\sign{(b-y)}e^{-|b-y|}\right]}_{S(y)}f(y)dy=F(b)-F(a).
\ee

Assume that $F(b)=F(a)=0$. If this is true, then we reach to a contradiction between \eqref{3.0.2} and \eqref{3.0.3}. The proof that $F(b)=F(a)=0$ is given in the Proposition \ref{prop3.3} below.\hspace{3.5cm}$\square$

\begin{proposition}\label{prop3.3}
Let $F$ be the function given above. Then $F\big|_{[a,b]}\equiv0$. In particular, $F(a)=F(b)=0$.
\end{proposition}

\begin{proof} We note that the equation in \eqref{1.0.2}, with $\gamma=-2\omega$, satisfies the identity
$$
m_t+(1-\p_x^2)(uu_x+2\omega u_x)+\p_x\left(u^2+\f{u_x^2}{2}\right)=\Lambda^2\left[u_t+(uu_x+2\omega u_x)+\p_x\Lambda^{-2}\left(u^2+\f{u_x^2}{2}\right)\right].
$$
Therefore, on the solutions of the equation in \eqref{1.0.2}, we have
$$
\p_x\Lambda^{-2}\left(u^2+\f{u_x^2}{2}\right)(t,x)=-\Bigg(u_t+(u+2\omega)u_x\Bigg)(t,x).
$$

We remember that $u$ vanishes on $\Omega$. Evaluating the expression above at $t=t_0$ we have
\bb\label{3.0.4}
F(x)=\p_x\Lambda^{-2}\left(u^2+\f{u_x^2}{2}\right)(t_0,x)=-\Bigg(u_t+(u+2\omega)u_x\Bigg)(t_0,x)
\ee
and then, for any $x\in[a,b]$, we have $F(x)=0$.
\end{proof}

{\bf Proof of Theorem \ref{teo1.3}.} Assume that $u$ has compact support. Then there is an open set $\emptyset\neq\Omega\nsubseteq\supp{u}$ such that $\left.u\right|_{\Omega}\equiv0$. The proof follows the same steps of the previous demonstration $3\Rightarrow 2$ and for this reason is omitted. 

\section{The periodic case}\label{sec5}

We present a more ``to the point'' demonstration of theorems \ref{teo1.4} and \ref{teo1.5}. We point out that this demonstration can also give a different proof to Theorem \ref{teo1.3}.

We begin with noticing the following remarks about the periodic problem \eqref{1.0.3}:
\begin{remark}\label{rem4.1}
Lemma \ref{lema2.1} assures the existence of a local solution $u\in C^0\left([0,T);H^s(\mathbb{S})\right)\cap C^1\left([0,T);H^{s-1}(\mathbb{S})\right)$.
\end{remark}

\begin{remark}\label{rem4.2}
The periodic problem has the functional
$$
{\cal H}(t):=\f{1}{2}\int_\mathbb{S}\left[u(t,x)^2+u_{x}(t,x)^2\right]dx=\f{1}{2}\int_0^1\left[u(t,x)^2+u_{x}(t,x)^2\right]dx=\f{1}{2}\|u(t,\cdot)\|^2_{H^1(\mathbb{S})}
$$
independent of time.
\end{remark}

\begin{remark}\label{rem4.3}
If $u:[0,T)\times\R\rightarrow\R$ is compactly supported, then there exists an open and connected set $\Omega\subsetneqq[0,T)\times\R$ such that $\left.u\right|_{\Omega}\equiv0$.
\end{remark}

\begin{remark}\label{rem4.4}
In line with Remark \ref{rem4.3}, without loss of generality we can assume the existence of $t_0\in(0,T)$ and $0<a<b<1$ such that $\{t_0\}\times[a,b]\subseteq\Omega$.
\end{remark}

\begin{remark}\label{rem4.5}
If $f$ is a non-negative, periodic and continuous function, then
$$(\Lambda^{-2}f)(x)=\int_{0}^1\f{\cosh{(x-y-\lfloor x-y \rfloor-1/2)}}{2\sinh(1/2)}f(y)dy\geq0.$$
In particular, $\Lambda^{-2}f=0$ if and only if $f\equiv0$.
\end{remark}

Let $u$ be a solution of \eqref{1.0.3} with compact support. By Remark \ref{rem4.3} there exists an open and connected set $\Omega$ such that $u$ vanishes on it. We note that $u$ then satisfies the conditions of Theorem \ref{teo1.5} and, accordingly, it is enough to prove the latter result.

Let $t_0$, $a$ and $b$ be the numbers satisfying the condition in Remark \ref{rem4.4}, $f(x)=u(t_0,x)^2+u_x(t_0,x)^2/2$ and $F(x)=(\p_x\Lambda^{-2}f)(x)$.
By Remark \ref{rem4.5}, $F\equiv0$ if and only if $f\equiv0$. 

We note that $F$ satisfies \eqref{3.0.4}, but now the action of the operator $\Lambda^{-2}$ is given as in Remark \ref{rem4.5}. Then we are forced to conclude that $F(b)=F(a)=0$. By the Fundamental Theorem of Calculus jointly with the identity $\p_x^2\Lambda^{-2}=\Lambda^{-2}-1$, the fact that $f(x)=0$, for all $x\in[a,b]$, we have
$$
0=F(b)-F(a)=\int_a^b F'(x)dx=\int_a^b(\Lambda^{-2}f)(x)dx.
$$

By Remark \ref{rem4.5} we conclude that $f\equiv0$ and, therefore, it is a foregone conclusion that $u(t_0,x)=0$. This implies that ${\cal H}(t_0)=0$ and the result is then a consequence of the conservation of the energy of the solutions. In particular, $u_0\equiv0$.

Note, however, that $u_0$ might be compactly supported and not necessarily be zero. In this case, $u(t,\cdot)$, $t>0$, loses the last property.

\section{Application to weakly dissipative versions of the equation}\label{sec6}

Here we show that theorems \ref{teo1.1}--\ref{teo1.5} can be applied to a large class of equations of the CH type. We begin with observing that our results are trivially applicable to the CH equation because if we take $\omega=\gamma=0$ in \eqref{1.0.2} we obtain
\bb\label{4.0.1}
\left\{
\ba{l}
m_t+um_x+2u_xm=0\\
\\
u(0,x)=u_0(x),
\ea
\right.
\ee
which is nothing but an initial value problem involving the CH equation.

Several works \cite{novjmp2013,wei,wujde2009,zhangjmp2015} have dealt with the equation
\bb\label{5.3.1}
\tilde{u}_t-\al^2\tilde{u}_{txx}+ \tilde{\gamma}(\tilde{u}_x-\al^2 \tilde{u}_{xxx})+3\tilde{u}\tilde{u}_x+\tilde{\lambda}(\tilde{u}-\al^2\tilde{u}_{xx})=\al^2(2\tilde{u}_x\tilde{u}_{xx}+\tilde{u}\tilde{u}_{xxx}).
\ee

Defining $\tilde{u}(t,x)=u(t/\al, x/\al)$, $\lambda:=\tilde{\lambda}\al$ and $\gamma:=\tilde{\gamma}\al$ we can transform \eqref{5.3.1} in the equation 
\bb\label{5.3.2}
u_t-u_{txx}+\gamma( u_x-u_{xxx})+3uu_x+\lambda(u-u_{xx}) =(2u_xu_{xx}+uu_{xxx}).
\ee

The parameter $\lambda$ is assumed to be non-negative. If $\lambda=0$ we recover the equation \eqref{1.0.2} with $\omega+\gamma/2=0$ whereas for $\lambda>0$ we have a sort of DGH equation with non-conservative energy functional, as we shall discuss below.

Let us consider the transformation (see \cite{lenjde2013})
\bb\label{5.3.3}
u(t,x)=e^{-\lambda t}u\left(\f{1-e^{-\lambda t}}{\lambda},x\right)=:e^{-\lambda t}v(\tau,\chi).
\ee
Let $n:=e^{-\lambda t}m$. A straightforward calculation yields:
\bb\label{5.3.4}
\ba{lll}
m=e^{-\lambda t}n,& u_t=-\lambda e^{-\lambda t}v+e^{-2\lambda t}v_\tau,& u_x=e^{-\lambda t}v_\chi,\\
\\
&m_t=-\lambda e^{-\lambda t}n+e^{-2\lambda t}n_\tau,&m_x=e^{-\lambda t}n_\chi,
\ea
\ee
where in \eqref{5.3.4} $(\cdot)_\tau$ and $(\cdot)_\chi$ denotes derivatives with respect to the first and the second arguments of the functions involved. Substituting \eqref{5.3.4} and \eqref{5.3.3} into \eqref{5.3.2} we get
\bb\label{5.3.5}
n_\tau+v n_\chi+2v_\chi n+\gamma \eta_\chi=0.
\ee
Moreover, $u_0(x)=u(0,x)=v(0,x)$ and, therefore, the equation \eqref{5.3.2} is transformed, through the change \eqref{5.3.3}, into the DGH equation with $\gamma=-2\omega$. As a consequence, theorems \ref{teo1.1}--\ref{teo1.5} can also be applicable to this case. 

\section{Conserved quantities and unique continuation for solutions}\label{sec7}

In this section we make an in-depth investigation about unique continuation for solutions of the DGH equation. We restrict our analysis to the real case. The situation in the periodic case can be recovered following the same steps as done in Section \ref{sec5} taking the corresponding particularities into account.

In view of our previous results, we restrict our analysis when $\gamma=-2\omega$. With this choice, we can rewrite \eqref{1.0.1} as (see the demonstration of Proposition \ref{prop3.3})
\bb\label{7.1}
u_t+(u+2\omega)u_x=-\p_x\Lambda^{-2}\left(u^2+\f{u_x^2}{2}\right).
\ee

Also, we recall again that the DGH equation, with the restrictions above, has the following conserved quantities
\bb\label{7.2}
{\cal H}(t)=\int_\R\left[u(t,x)^2+u_x(t,x)^2\right]dx
\ee
and
\bb\label{7.3}
{\cal H}_1(t)=\int_\R u(t,x)dx.
\ee

We shall assume mild conditions on a given solution $u$ of \eqref{7.1}. Our assumptions are:

\begin{enumerate}
    \item[{\bf C1}] The solution $u$ is defined on $[0,T)\times\R$, for some $T>0$;
    \item[{\bf C2}] $u$ is a solution conserving \eqref{7.2};
    \item[{\bf C3}] $u$ is a solution conserving \eqref{7.3};
    \item[{\bf C4}] For each $t$ fixed,  $\p_x\Lambda^{-2}\left(u^2+u_x^2/2\right)(t,\cdot)\in C^1$;
    \item[{\bf C5}] There exists numbers $t_0,\,t_1,\,x_0,\,x_1$, with $0<t_0<t_1<T$ and $x_0<x_1$, such that $\Omega:=(t_0,t_1)\times(x_0,x_1)\subseteq [0,T)\times\R$ and $u(t,x)=0$, for all $(t,x)\in\Omega$.
\end{enumerate}

We would like to point out some observations about the conditions above:
\begin{remark}\label{rem7.1}
Condition ${\bf C1}$ requires the existence of a solution, but not its uniqueness. Besides, note that we might eventually have $T=\infty$, but this is unessential.
\end{remark}

\begin{remark}\label{rem7.2}
Conditions ${\bf C2}$ and ${\bf C3}$ enclose a large class of solutions of the DGH equation. This is easily explained by the following formal identities ({\it e.g}, see \cite[Theorem 3.1]{raspajde}):
\bb\label{7.4}
\ba{l}
\ds{\p_t u+\p_x\left(\f{3}{2}u^2-uu_{xx}-\f{u_x^2}{2}-u_{tx}+2\omega u+\gamma u_{xx}\right)}\\
\\
\ds{=u_t-u_{txx}+2\omega u_x+3uu_x+\gamma u_{xxx}-2u_xu_{xx}-uu_{xxx}}
\ea
\ee
and
\bb\label{7.5}
\ba{l}
\ds{\p_t\left(\f{u^2+u_x^2}{2}\right)+\p_x\left(u^3-u^2u_{xx}-uu_{tx}-\f{\gamma}{2}u_x^2+\gamma uu_{xx}+\omega u^2\right)}\\
\\
\ds{=u\left(u_t-u_{txx}+2\omega u_x+3uu_x+\gamma u_{xxx}-2u_xu_{xx}-uu_{xxx}\right)}.
\ea
\ee

Therefore, if $u,u_x,u_{xx},u_{tx}\rightarrow0$ as $x\rightarrow\pm\infty$, from \eqref{7.4} and \eqref{7.5} we obtain, respectively,
$$
\f{d}{dt}\int_\R u=-\left(\f{3}{2}u^2-uu_{xx}-\f{u_x^2}{2}-u_{tx}+2\omega u+\gamma u_{xx}\right)\Big|_{-\infty}^{+\infty}=0
$$
and
$$
\f{d}{dt}\int_\R \left(\f{u^2+u_x^2}{2}\right)=-\left(u^3-u^2u_{xx}-uu_{tx}-\f{\gamma}{2}u_x^2+\gamma uu_{xx}+\omega u^2\right)\Big|_{-\infty}^{+\infty}=0.
$$

It is a foregone conclusion that any solution with mild regularity decaying to $0$ at infinity, jointly with its derivatives up to second order, would fall in the conditions {\bf C2} and {\bf C3}. In particular, solutions $u\in H^s(\R)$, for $s$ sufficiently large, satisfy these conditions. Note that we have just shown the reasons for which \eqref{7.2} and \eqref{7.3} are conserved quantities for the DGH equation. For the periodic case we can recover these results requesting that $u$ and its derivatives have the same values at $(z,t)$, where $z\in\mathbb{Z}$.
\end{remark}

\begin{remark}\label{rem7.3}
In condition {\bf C5} we consider an open rectangle. At first sight, this seems to be very restrictive. This is only an apparent contradiction, since any open set contains at least one (and, therefore, infinite) open rectangles.
\end{remark}

\begin{proposition}\label{prop7.1}
Let $u$ be a solution of \eqref{7.1} satisfying the conditions {\bf C1}, {\bf C4} and {\bf C5}; $a$ and $b$ real numbers such that $a<b$ and $a,b\in(x_0,x_1)$; and, for each $t\in(t_0,t_1)$ fixed, let us define $f_t:\R\rightarrow\R$ by $f_t(x)=\Lambda^{-2}\left(u^2+u_x^2/2\right)(t,x)$ and $F_t(x)=\p_x\Lambda^{-2}\left(u^2+u_x^2/2\right)(t,x)$. Then $\left. F_t\right|_{[a,b]}\equiv0$ as well as $ f_t\equiv0$.
\end{proposition}
\begin{proof}
Similar to Proposition \ref{prop3.3} and implication $3\Rightarrow2$ of Theorem \ref{teo1.2}. For this reason is omitted.
\end{proof}

Proposition \ref{prop7.1} has a beautiful geometrical meaning: 
Let us fix $t^\ast\in(t_0,t_1)$. The fact that the solution $u$ vanishes on the vertical segment $\{t^\ast\}\times[a,b]$ implies that $F_{t^\ast}$ vanishes on $[a,b]$ provided that $x_0<a<b<x_1$. Consequently, this forces $\Lambda^{-2}f(x)=0$, that is,
$$
\int_{\R} e^{-|x-y|}f(y)dy=0
$$
and, therefore, $f_{t^\ast}(x)=0$, for each $x\in\R$. Then, the fact that $u$ vanishes on the vertical segment $\{t^\ast\}\times[a,b]$ implies that the it vanishes on the set $\Omega_{t^\ast}:=\{(t^\ast,x);\,f_{t^\ast}(x)=0\}=\{t^\ast\}\times\R$, which is nothing but the unique vertical line containing the segment $\{t^\ast\}\times[a,b]$. Once noticing that the family $(f_t)_{t\in(t_0,t_1)}$ vanishes identically we conclude that the solution $u$ vanishes on the strip
$${\cal R}:=\bigcup_{t\in(t_0,t_1)}\Sigma_t=(t_0,t_1)\times\R$$
which is the a strip of width $t_1-t_0$ containing $\Omega$. Moreover, this strip is uniquely determined by the set $\Sigma$.

We observe that the ideas we brought from \cite{linares} enabled us to construct the strip above and show that the solution $u$ vanishes on ${\cal R}$. The question now is: Does the solution vanish outside ${\cal R}$? In order to address this question, we recall that if $\gamma=-2\omega$, then the DGH equation may have solutions that do not change their sign ({\it e.g.}, if the initial momentum of the DGH equation, with the restriction $\gamma=-2\omega$, does not change sign, then Lemma \ref{lema2.2} implies that $m$ and $u$ will have the same signs and they do not change as well). This fact, jointly with \eqref{7.3}, proves the following result:
\begin{proposition}\label{prop7.2}
If $u$ is a solution of the DGH equation that does not change and satisfies the condition ${\bf C2}$, then $\|u(t,\cdot)\|_{L^1(\R)}$ is conserved as long as the solution exists.
\end{proposition}

We can now give a characterisation of the vanishing solutions of the DGH equation that does not change their sign.

\begin{theorem}\label{teo7.1}
Let $u$ be a solution of the DGH equation such that its sign does not change. If $u$ satisfies the conditions ${\bf C1}$, ${\bf C3}-{\bf C5}$, then $u\equiv0$.
\end{theorem}
\begin{proof}
By Proposition \ref{prop7.1} we can find $t^\ast$ such that $u(t^\ast,x)=0$, for all $x\in\R$. If $u$ is non-negative, from \eqref{7.3} we conclude that ${\cal H}_1(t)=\|u(t,\cdot)\|_{L^1(\R)}$, while if $u$ is non-positive, we have ${\cal H}_1(t)=-\|u(t,\cdot)\|_{L^1(\R)}$. In any case, evaluating \eqref{7.3} at $t=t^\ast$ we conclude that $\|u(t^\ast,x)\|_{L^1(\R)}=0$. Since \eqref{7.3} is invariant in $t$, we conclude that $\|u(t,\cdot)\|_{L^1(\R)}$ for every $t$ for which the solution exits, which implies that $u(t,x)=0$, for all $(t,x)$.
\end{proof}

An even stronger result can be given right know:

\begin{theorem}\label{teo7.2}
Let $u$ be a solution of the DGH equation satisfying the conditions ${\bf C1}.\,{\bf C2},\,{\bf C4}$ and ${\bf C5}$. Then $u\equiv0$.
\end{theorem}

Theorem \ref{teo7.2} is a consequence of propositions \ref{prop3.1} and \ref{prop7.1} and, therefore, its proof is omitted.

We note that theorems \ref{teo7.1} and \ref{teo7.2} also have a geometrical meaning: it uses the conserved quantity to extend, for all $t$, the result known to a particular $t^\ast$, meaning that once we know the information on a vertical straight line, we can translate this property to any parallel line in which the solution is defined.

It also warrants observation that in spite of the fact that the family $(f_t)_{t\in(t_0,t_1)}$ vanishes, our demonstration of theorems \ref{teo7.1} and \ref{teo7.2} only really needs the existence of a single value $t^\ast$ for which $f_{t^\ast}$ is null. On the other hand, it is immediate from our construction that if there is one function satisfying this condition, we can then construct a continuous one-parameter family of functions having the same property.

\section{Discussion}\label{sec8}

Theorem \ref{teo1.1} implies that for any $t>0$ such that the solution of the initial value problem \eqref{1.0.2} exists (as assured by Lemma \ref{lema2.1}), if $m_0(x)+\omega+\gamma/2$ has compact support, then this property is extended  for the function $m(t,\cdot)+\omega+\gamma/2$. Differently from the DGH equation, in the CH case the function $m(t,\cdot)$ has compact support provided that $m_0$ shares the same property, see \cite[Sec.~II]{constjmp} and \cite[Proposition~1]{henry}. For the CH equation it is not possible to guarantee that the support of $m(t,\cdot)$ is the same for each $t>0$. Similarly, for the DGH equation we cannot affirm that $m(t,\cdot)+\omega+\gamma/2$ is supported on the same set for each $t>0$. On the contrary, the proof of the theorem indicates that the support may vary while $t$ changes.

Theorem \ref{teo1.2} says that if a solution of the problem \eqref{1.0.2}, with initial data $u_0\in H^4(\R)$, vanishes on a non-empty set $\Omega\subseteq[0,T)\times\R$, then it not only vanishes on its whole domain, but also the own solution is uniquely extended to any time $t>0$ to the identically vanishing solution.

The key to prove this result was to fix a certain time and use the non-locality of the convolution to prove that if a solution vanishes on a finite vertical segment on the $(t,x)-$plane, then it vanishes over the whole straight line containing such segment and, therefore, in view of the conservation of energy of the solutions of \eqref{1.0.2}, it vanishes everywhere the solution exists. 

The proof of the implication $3\Rightarrow 2$ in Theorem \ref{teo1.2} is based on \cite[Theorem~1.3]{linares}. We observe, however, that our demonstration has some nuances when compared with that aforesaid. For example, differently from \cite{linares}, we used the conservation of energy to conclude that the solution vanishes on its entire domain, as previously mentioned.

Once this fact is established, we can show that the unique compactly supported solution of \eqref{1.0.2} is $u\equiv0$. Here we assume that the initial data belongs to $H^4(\R)$. This fact is well known for the CH equation, see \cite[Sec.~II]{constjmp} and \cite[Theorem~1]{henry}, but as far as the author knows, it has not been reported to the periodic DGH equation neither to weakly dissipative versions of this equation, both in the real and periodic cases. For the real case of the DGH equation, there is a continuation result due to Zhou and Guo \cite{zhou} implying that if a solution of the (real) DGH equation is compactly supported solution, then it necessarily is the identically vanishing one. However, their approach is completely different from ours.  

Although this work is strongly influenced by the references \cite{constjmp,henry}, it is rather different of them even if directly applied to the CH equation, since we firstly proved that if the solution vanishes on some non-empty open set, then it vanishes identically and the property of having compact support is implied by this fact.

We observe that if $u\not\equiv0$ is a solution of \eqref{1.0.2} with $u_0\in H^4(\mathbb{E})$, then it cannot be compactly supported for all $t>0$. However, it might occur that $u_0\neq0$ be compactly supported and this is the unique value of $t$ such that the solution can have this property. If such situation occurs, then $u(\cdot,x)$ loses it instantly, as observed by Constantin \cite[Remark at the end of Section II]{constjmp}. On the other hand, provided that $u(t,x)\neq0$ for some $t$, then surely $u_0$ cannot vanish identically due to the conservation of energy. 

Note that we extended our results in an intrinsic way for weakly dissipative models involving the CH and the DGH equations. Our idea here is the use of the transformation discovered in Lenells and Wunsch \cite{lenjde2013} to extend the results regarding well-posedness of the CH and other related equations to their corresponding weakly dissipative counter-parts. This enables us to reduce the unique continuation property of solutions and the study of the compactly supported ones of the equation $$u_t-\al^2u_{txx}+2\omega u_x+3uu_x+\gamma u_{xxx}+\lambda(u-u_{xx})=\al^2(2u_xu_{xx}+uu_{xxx})$$
to the DGH equation.

Last, but not least, two {\it sine qua non} ingredients to prove our main results is the idea introduced in \cite{linares} in conjunction with the use of conserved quantities, which gives a new approach not only to study compactly supported solutions of the DGH equation, but also to give a different viewpoint for proving unique continuation for solutions of the DGH equation. This idea, indeed, is consistent with the Physics behind the phenomena:

\begin{itemize}
\item The conserved quantity \eqref{7.2} corresponds to the energy of the solutions of the DGH equation. In our approach, we show that if the solution vanishes on an open set, then Proposition \ref{prop7.1} says that the solution vanishes on an entire vertical straight line. From the point of view of conserved quantities, this means that for a given time $t^\ast$, the solution vanishes, which means that the energy at that time vanishes as well. Since it is conserved, then it must be zero everywhere.

\item On the other hand, the conserved quantity \eqref{7.3} correspond to the conservation of mass for \eqref{1.0.1}. Then, Proposition \ref{7.1} essentially says the if $u$ vanishes on an open set, then we can find a vertical line such that the mass vanishes and, since it is conserved, we do not have mass anywhere.
\end{itemize}

We note that the last two paragraphs also have meaning in the case of periodic solutions. The interpretation can be easily retrieved by replacing vertical segment by arc segment, and vertical line by circle. Then, the conserved quantities \eqref{7.2} or \eqref{7.3} translates to all circles in which the solution is defined the information obtained in a given circle at $t=t^\ast$. 

To conclude this section, we note that the DGH equation with $\gamma=-2\omega$ can be written in the form
\bb\label{8.1}
u_t+g+\p_x\Lambda^{-2} h=0,
\ee
where $g=uu_x+2\omega u_x$ and $h=u^2+u_x^2/2$.

The ideas presented here are applied to the class \eqref{8.1} for arbitrary $g$ and $h$, provided that:
\begin{enumerate}
    \item\label{c1} $g$ is a continuous function of $u$ and its derivatives, and does not have explicit dependence on $t$ and $x$. Moreover, we request that $g$ vanishes whenever (all of) its arguments vanish. Note that the dependence of $g$ is not specified because it depends on each equation;
    \item\label{c2} $h$ is a function satisfying the following conditions: $|h|>0$, its sign does not change, and $\p_x\Lambda^{-2}h$ is continuous, for each $t$ fixed. Likewise the function $g$, we do not specify the dependence of $h$.
    \item\label{c3} The equation must have a conserved quantity
    $${\cal H}(t)=\int_\R f dx,$$
    where the integrand $f$ is a function that does not change sign and $f=0$ only when $u(t,x)=0$. Note that eventually we may have some conditions on $u$ to ensure that $f$ satisfies the required property.
\end{enumerate}

Note that if we replace $u^2+u_x^2/2$ by $h$ in Proposition \ref{prop7.1}, we would then obtain the same conclusion. This was firstly observed in \cite{linares}. However, in this reference the authors requested that $h>0$. There is no reason to exclude the case $h<0$. What is really necessary is that $h$ is a non-vanishing function whose sign does not change. Also, \ref{c1} and \ref{c2} enable us to construct a strip ${\cal R}$ in which the solution vanishes.

Condition \ref{c3} is what makes possible the conclusion that the solution $u$ vanishes outside the strip. In fact, if for some $t^\ast$ we have $(t^\ast,\cdot)\in{\cal R}$, then ${\cal H}(t^\ast)=0$. Since ${\cal H}(\cdot)$ is invariant, we conclude that ${\cal H}(t)=0$, for all $t$. The fact that $f$ is either non-negative or non-positive yields $f\equiv0$ and since $f=0$ only when $u=0$, we are forced to conclude that $u(t,\cdot)\equiv0$, for each $t$ the solution exists.

\section{Conclusion}\label{sec9}

In this work we show that if the DGH equation has a compactly supported solution for some $t>0$, then it necessarily vanishes.

In order to do this, we bring use the ideas in \cite{linares} for unique continuation for the solutions of the Camassa-Holm with the conserved quantities of the DGH equation. As a consequence, we give an alternative approach for proving unique continuation for solutions of equations of the Camassa-Holm type. 

Finally, we extended the results of theorems \ref{teo1.1}--\ref{teo1.5} to a large class of equations of the CH and DGH type, as shown in the Section \ref{sec5}. We also explore our idea for unique continuation for solutions of the DGH equation using two different conserved quantities to show that if we can find an open set in which the solution vanishes, then it vanishes everywhere.

\section*{Acknowledgements}

The author is thankful to Dr. P. L. da Silva for the stimulating discussions about the DGH equation, which motivated the investigation of the results reported here. Also, Professor G. Ponce is thanked for having driven my attention to Reference \cite{himcmp}. I would like to thank V. H. C. Freire for the careful reading of the manuscript. Finally, CNPq is also thanked for financial support, through grant nº 404912/2016-8.

\end{document}